\documentclass[12pt,reqno]{article}
\usepackage[letterpaper,margin=1.6cm]{geometry}
\usepackage{latexsym, amsfonts, amsthm,amssymb,amscd,amsmath,makeidx}
\usepackage{enumerate}
\usepackage[normalem]{ulem}
\usepackage{exscale}
\usepackage{overpic} 
\usepackage{color} 
\usepackage{graphicx}
\usepackage{overpic}  
\usepackage{bm}
\usepackage{comment}

\usepackage{mathrsfs}
\usepackage{amssymb}
\usepackage{mathtools}

\usepackage{tikz}
\usepackage{blindtext}
\usepackage{bbm}
\usepackage{fancyhdr} 
\usepackage{mathrsfs}
\usepackage{wrapfig}
\usepackage{hyperref}
\usepackage{cleveref}
\usepackage{cases}

\hypersetup{
    colorlinks,
    citecolor=black,
    filecolor=black,
    linkcolor=blue,
    urlcolor=black
}

\definecolor{DarkGreen}{rgb}{0.2,0.6,0.2}

\def\eps{\varepsilon}
\def\Om{\Omega}\def\om{\omega}

\def\Ind#1{{\mathbbmss 1}_{_{\scriptstyle #1}}}

\def\lra{\longrightarrow}

\def\ua{\uparrow}
\def\da{\downarrow}

\def\wh{\widehat}
\def\wt{\widetilde}

\def\ignore#1{}

\def\bR{{\mathbb R}}

\def\bZ{\mathbb Z}
\def\bN{\mathbb N}

\def\bP{{\mathbb P}}
\def\bE{{\mathbb E}}

\numberwithin{equation}{section}

\usetikzlibrary{positioning,calc,cd}

\allowdisplaybreaks[4]

\def\cF{{\mathscr F}}

\newtheorem{theorem}{Theorem}[section]
\newtheorem{proposition}[theorem]{Proposition}
\newtheorem{lemma}[theorem]{Lemma}
\newtheorem{corollary}[theorem]{Corollary}

\theoremstyle{definition}

\newtheorem{example}[theorem]{Example}

\newtheorem{remark}[theorem]{Remark}

\def\Ind#1{{\mathbbmss 1}_{_{\scriptstyle #1}}}
\def\lra{\longrightarrow}
\def\eps{\varepsilon}
\def\<{\langle}
\def\>{\rangle}
\def\wt#1{\widetilde{#1}}

 \def\lra{\longrightarrow}

 \def\ua{\uparrow}
 \def\da{\downarrow}
 
 \def\wh{\widehat}
 \def\wt{\widetilde}

\begin{document}

\title{A limit theorem for Bernoulli convolutions\\ and the $\Phi$-variation of functions in the Takagi class}
\author{ Xiyue Han\thanks{Department of Statistics and Actuarial Science, University of Waterloo. E-mail: {\tt xiyue.han@uwaterloo.ca}} \and Alexander Schied\setcounter{footnote}{6}\thanks{Department of Statistics and Actuarial Science, University of Waterloo. E-mail: {\tt aschied@uwaterloo.ca}}
	 \and
	Zhenyuan Zhang\setcounter{footnote}{3}\thanks{
		 Department of Mathematics, Stanford University. E-mail: {\tt zzy@stanford.edu}
	 		\hfill\break The authors gratefully acknowledge financial support  from the
 Natural Sciences and Engineering Research Council of Canada through grant RGPIN-2017-04054}}
        
        \date{\normalsize December 12, 2021}

\vspace{-1cm}

\maketitle

\begin{abstract}We consider a probabilistic approach to compute the Wiener--Young $\Phi$-variation of fractal functions in the Takagi class. Here, the $\Phi$-variation is understood as a generalization of the quadratic variation or, more generally, the $p^{\text{th}}$ variation of a trajectory  computed along the sequence of dyadic partitions of the unit interval.   The functions $\Phi$ we consider  form a very wide class of  functions that are regularly varying at zero. Moreover, for each such function $\Phi$,   our results provide in a straightforward manner  a large and tractable class of functions that have nontrivial and linear $\Phi$-variation.  As a corollary, we also construct stochastic processes whose sample paths have nontrivial, deterministic, and linear $\Phi$-variation for each function $\Phi$ from our class. The proof  of our main result relies on a limit theorem for certain sums of Bernoulli random variables that converge to an infinite Bernoulli convolution.  \end{abstract}


\medskip


\section{Introduction}

Probabilistic models  of continuous random trajectories typically rely on 
solutions to stochastic differential equations or, more generally, on continuous semimartingales. The huge success of those models is in no small part due to  the fact that they can be analyzed by means of It\^o calculus. 
% The ultimate reason why It\^o calculus works for continuous semimartingales is that their sample paths admit a (typically nontrivial) quadratic variation, which exists almost surely  along a fixed refining sequence of partitions. 
It\^o calculus is however not restricted to the sample paths of a semimartingale:  F\"ollmer \cite{FoellmerIto} showed that every continuous function that admits a continuous quadratic variation along a fixed refining sequence of partitions can be used as an integrator for It\^o integration and that a corresponding It\^o formula holds. As shown in 
\cite{Gantert,MishuraSchied,SchiedTakagi}, this includes many functions previously studied in fractal analysis and geometry.

Cont and Perkowski \cite{ContPerkowski} recently extended F\"ollmer's pathwise It\^o formula to integrators with higher than quadratic variation, the so-called $p^{\text{th}}$ variation. While quadratic variation is based on the squared increments of a function, $p^{\text{th}}$ variation is based on the $p^{\text{th}}$ power of the increments, where $p\ge1$.  In the realm of stochastic processes, the sample paths of fractional Brownian motion with Hurst index $1/p$ constitute the best-known class of such trajectories. At the same time, recent applications, such as the theory of rough volatility models based on the seminal paper \cite{GatheralRosenbaum} by Gatheral et al., have intensified the interest in creating new models for \lq\lq rough" phenomena beyond the case of quadratic variation. 
Also in this strand of literature, the focus has so far been on models with finite and nontrivial $p^{\text{th}}$ variation. 

 From a mathematical point of view, however, there is no reason why attention should be restricted to variation based on power functions. A natural and much broader concept is given by the notion of  $\Phi$-variation in the Wiener--Young sense. Here,  for a function $\Phi:[0,\infty)\to[0,\infty)$,  the $\Phi$-variation of a continuous function $f:[0,1]\to\bR$  along the dyadic partitions is given by
 \begin{equation}\label{Phi var intro def}
\<f\>^\Phi_t:=\lim_{n\uparrow\infty}\sum_{k=0}^{\lfloor t2^n\rfloor}\Phi\big(|f((k+1)2^{-n})-f(k2^{-n})|\big),\qquad 0\le t\le1, \end{equation}
provided that the limit exists for all $t$. 
Yet, most authors still focus exclusively on the case of power variation. Notable exceptions are the paper \cite{MarcusRosenPhi} by Marcus and Rosen, in which the $\Phi$-variation of  Gaussian processes with stationary increments and local times of symmetric L\'evy processes is obtained, and K\^ono \cite{KonoOscillation}. Note that the $\Phi$-variation along a fixed refining sequence of partitions as defined in \eqref{Phi var intro def} typically differs from the $\Phi$-variation defined as a supremum taken over \emph{all} finite partitions.  For Gaussian processes, the latter concept was studied, e.g., by Taylor \cite{Taylor} for $\Phi(x)=x^2$ and by Kawada and K\^ono \cite{KawadaKono} for more general functions $\Phi$.

In this paper, we use a probabilistic approach to explore a class of fractal trajectories $f$ that admit the $\Phi$-variation $\<f\>^\Phi_t=t$ for functions $\Phi$ of the form 
\begin{equation}\label{Phiq intro}
\Phi(x) := \begin{cases}0&\text{if $x=0$,}\\
x^{p}g\big({-p}\log_2 x\big)^{-p/2}&\text{if $0<x<1$},
\end{cases}
\end{equation}
where $p\in[1,\infty)$ and $g:[0,\infty)\to(0,\infty)$ is any regularly varying function\footnote{ Recall that a measurable and strictly positive function $g$, which is defined on some interval $[a,\infty)$, is called regularly varying (at infinity) with index $\varrho\in\bR$ if 
$g(\lambda x)/g(x)\to\lambda^\varrho$ as $x\ua\infty$ for all $\lambda>0$. If $\varrho=0$, the function $g$ is also called slowly varying. Throughout this paper, we will always assume that all regularly varying functions are defined on $[0,\infty)$. This assumption can be made without loss of generality, as one can always consider the function $x\mapsto g(a\vee x)$. \label{reg varying footnote}}. In doing so, we address several high-level issues. First, we illustrate that finite and nontrivial $p^{\text{th}}$ power variation, which has so far been the default paradigm for characterizing the roughness of trajectories, appears to be the exception rather than the rule.  In particular, one should expect $\Phi$-variation rather than power variation when studying continuous trajectories that do not arise as sample paths of a semimartingale. Second, for every function $\Phi$ of the form \eqref{Phiq intro}, we construct explicit examples of trajectories with finite linear $\Phi$-variation along the dyadic partitions. These examples can be used in further analyses involving \lq\lq rough" trajectories. Moreover, the idea underlying our construction of functions with prescribed $\Phi$-variation can probably be extended to more flexible models and into the realm of stochastic processes.

The class of functions $f$ we consider here belong to the so-called Takagi class, which was introduced by Hata and Yamaguti \cite{HataYamaguti} and motivated by Takagi's \cite{Takagi} celebrated example of a continuous nowhere differentiable function. This class was chosen  due to its richness and tractability and also because it is closely related to the Wiener--L\'evy construction of Brownian motion. The Takagi class consists of all functions $f:[0,1]\to \bR$ that admit an expansion of the form 
\begin{align}\label{vdw1}
f(t):=\sum_{m=0}^\infty \alpha_m\varphi(2^mt),\qquad t\in[0,1],
\end{align}
where $\varphi(t)=\min_{z\in\bZ}|z-t|$ is the tent map and  $\{\alpha_m\}_{m \in \bN_0}$ is a sequence of real numbers for which the series $\sum_{m=0}^\infty\alpha_m$ converges absolutely. The special choice $\alpha_m=a^m$ for some $a\in(-1,1)$ yields the class of so-called Takagi--Landsberg functions, which, for   $1/2<|a|<1$ and $p=-\log_{|a|}2$, have finite nontrivial and linear $p^{\text{th}}$ variation  as shown in \cite{MishuraSchied2}. For $|a|<1/2$ it is well known that $f$ is of finite total variation. The borderline case $a=1/2$ corresponds to the classical Takagi function, which is nowhere differentiable and hence not of finite total variation, even though $p=-\log_{1/2}2=1$. Instead, for $|a|=1/2$, the function $f$ has linear $\Phi$-variation, $\<f\>_t^\Phi=t$, for $\Phi(x)=x\sqrt{\pi/(-2\log_2 x)}$, as  shown by the authors in \cite{HanSchiedZhang1}.  Note that this function $\Phi$ is a special case of  \eqref{Phiq intro}. 
Our method also extends to an even richer class of functions for which each scaled phase  of the tent map in \eqref{vdw1} is multiplied with an arbitrary sign (see \eqref{flexible class} for details). This latter class can be fitted to empirical time series  and gives also rise to a class of stochastic processes if the signs are chosen randomly.

Our approach to the $\Phi$-variation of functions $f$ in the Takagi class depends on the parameter $p$ in \eqref{Phiq intro}. For $p=1$, our argument extends the approach via the central limit theorem for the Rademacher functions in the Faber--Schauder development of $f$ as presented in \cite{HanSchiedZhang1}. For $p>1$, however, a new limit theorem is needed, which might be of independent interest. If   $\{\beta_m\}_{m\in \bN_0}$ is a sequence of real numbers and $Y_1,Y_2,\dots$ is an i.i.d.~sequence of  symmetric Bernoulli random variables,  we investigate the limit of the \lq\lq convolutions"
$$\frac1{s_n}\sum_{m=1}^{n}\beta_{n-m}Y_m \qquad\text{ where}\qquad s_n^2:=\sum_{m=0}^{n-1}\beta_m^2.
$$
We show that if there exist $q>0$, $b>1$, and a slowly varying function $\ell$ such that 
\begin{equation}\label{snL2n condition2 intro}\lim_{n\to\infty}\frac{s_n^2}{2^{2qn}\ell(b^n)}=1,
\end{equation}
then 
\begin{equation}\label{Z def eq intro}
\frac1{s_n}\sum_{m=1}^{n}\beta_{n-m}Y_m\lra \sqrt{1-2^{-2q}}\sum_{m=1}^\infty 2^{-qm}Y_m\qquad \text{in $L^\infty$.}
\end{equation}
Note that the law of the limiting random variable $Z$ is a scaled version of the infinite Bernoulli convolution with parameter $2^{-q}$. In Section \ref{limit section}, we will present an extended version of this result in which the $Y_m$ are neither required to be independent nor identically distributed. This approach to the $\Phi$-variation via the condition \eqref{snL2n condition2 intro} and the limiting result \eqref{Z def eq intro} is to some extent motivated by Gladyshev's theorem \cite{Gladyshev} (see also  the book \cite{MarcusRosen} by Marcus and Rosen for a comprehensive survey). It extends previous work in \cite{MishuraSchied2,SchiedZZhang,HanSchiedZhang1}.

The paper is structured as follows. In \Cref{limit section}, we derive the convergence result \eqref{Z def eq intro}
 and the central limit theorem that will be needed for  our results on the $\Phi$-variation of functions in the Takagi class. Our main result on this topic, \Cref{Theorem on Phi-variation}, is stated in \Cref{Phi variation section}. The proof of \Cref{Theorem on Phi-variation} is given in \Cref{proofs section}.

 \section{A limit theorem for Bernoulli-type convolutions}\label{limit section}

For some $p\in[1,\infty]$, let $Y_1,Y_2,..$ be a sequence of random variables in $L^p:=L^p(\Om,\cF,\bP)$, where $(\Omega,\cF,\bP)$ is a given probability space. We assume for simplicity that the $L^p$-norms of $Y_m$ are uniformly bounded. Let furthermore   $\{\beta_m\}_{m\in \bN_0}$ be a sequence of real numbers  and consider the random variables
\begin{equation*}
Z_n:=\sum_{m=1}^{n}\beta_{n-m}Y_m,\qquad n\in\bN.
\end{equation*}
We are first concerned with the limit of
$$\frac1{s_n}Z_n \qquad\text{where}\qquad s_n^2:=\sum_{m=0}^{n-1}\beta_m^2
$$
as $n\ua\infty$.  Our corresponding result uses the concept of a slowly varying function, which was recalled in the footnote on page~\pageref{reg varying footnote}.

\begin{theorem}\label{conv thm} Suppose that $\beta_m\ge0$ for all $m$ and that there exist $q>0$, $b>1$, and  a slowly varying function $\ell$ such that 
\begin{equation}\label{snL2n condition2}\lim_{n\to\infty}\frac{s_n^2}{2^{2qn}\ell(b^n)}=1.
\end{equation}
Then 
\begin{equation*}\label{Z def eq}
\frac1{s_n}Z_n\lra Z:=\sqrt{2^{2q}-1}\sum_{m = 1}^\infty 2^{-qm}Y_m\qquad \text{in $L^p$.}
\end{equation*}
\end{theorem}

\begin{proof}[Proof of \Cref{conv thm}]  For $\delta\in(0,\delta_0]$, denote
$$\lambda:=\sqrt{2^{2q}-1},\qquad \lambda^{+}_\delta  := \sqrt{2^{2q}\dfrac{(1+\delta)^2}{1-\delta}-1},\quad\text{and}\quad \lambda^{-}_\delta : =  \sqrt{2^{2q}\dfrac{(1-\delta)^2}{1+\delta}-1},$$ 
 where $\delta_0\in(0,\min\{1,q\})$ is chosen such that the arguments of the  square roots are always positive. Thus,
\begin{equation}\label{delta0}
2^{2q}(1-\delta)^2-(1+\delta)>0\qquad\text{for all $\delta\in(0,\delta_0]$.}
\end{equation}  Let $\eps>0$ be given. We first choose $K\in\bN$ such that 
\begin{equation}\label{lambda times sum eq}
\lambda\sum_{m=K+1}^\infty 2^{-(q-\delta_0)m}<\eps.
\end{equation}
Then we choose 
$\delta_1\in(0,\delta_0]$ such that 
\begin{equation}\label{1-2deltam eq}
1-2^{-\delta m}<\eps\qquad \text{for $m=0,\dots, K$ and all $\delta\in(0,\delta_1]$.}
\end{equation}
 Finally, we choose $\delta\in(0,\delta_1]$ such that 
\begin{equation}\label{lambdapm eq}
\bigg|\frac{\lambda_\delta^\pm}{\sqrt{1\mp\delta}}-\lambda\bigg|<\eps.
\end{equation}
Given this $\delta$, we use \eqref{snL2n condition2} to pick $n_0\in\bN$ such that
\begin{equation}\label{N0(eps) eq}
    \ell(b^n)2^{2qn}(1 - \delta) < s_n^2 < \ell(b^n)2^{2qn}(1 + \delta)\qquad\text{ for all $n > n_0$.}
\end{equation}
 Hence,
 \begin{equation}\label{betamsquared eq} 
  2^{2q(m+1)}\big((1-\delta)\ell(b^{m+1}) - (1 + \delta)\ell(b^{m})2^{-2q}\big) < \beta_m^2<  2^{2q(m+1)}\big((1+\delta)\ell(b^{m+1}) - (1 - \delta)\ell(b^{m})2^{-2q}\big).
 \end{equation}
  Since $\ell$ is slowly varying, 
   we can choose $n_1\ge n_0$ such that 
\begin{equation*}
1-\delta<\frac{\ell(b\cdot b^{m})}{\ell(b^{m})}<1+\delta\qquad\text{for all $m\ge n_1$.}
\end{equation*}
 Combining this with \eqref{delta0}, we see that the left-hand side of \eqref{betamsquared eq} 
 is strictly positive for $m\ge n_1$. Hence, we may take square roots to get 
 \begin{equation}\label{an2n ineq}
    2^{q(m+1)}\sqrt{(1-\delta)\ell(b^{m+1}) - (1 + \delta)\ell(b^{m})2^{-2q}} < \beta_m<  2^{q(m+1)}\sqrt{(1+\delta)\ell(b^{m+1}) - (1 - \delta)\ell(b^{m})2^{-2q}}
    \end{equation}
for all $m\ge n_1$. Moreover, with \eqref{N0(eps) eq} we get that  for  $n>m \ge n_1$,
\begin{align*}
    \dfrac{\beta_{m}}{s_{n}} &< \dfrac{2^{q(m+1)}\sqrt{(1+\delta)\ell(b^{m+1}) - (1 - \delta)\ell(b^{m})2^{-2q}}}{\sqrt{(1-\delta)2^{2qn}\ell(b^n)}} = 2^{q(m-n)}\sqrt{\frac{\ell(b^m)}{\ell(b^n)}\bigg(2^{2q}\dfrac{1+\delta}{1-\delta}\dfrac{\ell(b^{m+1})}{\ell(b^m)}-1\bigg)}\\
       &< 2^{q(m-n)}\sqrt{\dfrac{\ell(b^m)}{\ell(b^n)}}\lambda^{+}_\delta .
\end{align*}
Similarly, we conclude a corresponding lower bound and obtain that for $n>m\ge n_1$,
\begin{equation}
    2^{-q(n-m)}\sqrt{\dfrac{\ell(b^m)}{\ell(b^n)}}\lambda^{-}_\delta  < \dfrac{ \beta_m }{s_n}< 2^{-q(n-m)}\sqrt{\dfrac{\ell(b^m)}{\ell(b^n)}}\lambda^{+}_\delta .\label{qk}
\end{equation}

 The Potter bounds (see, e.g., Theorem 1.5.6 in \cite{BinghamGoldieTeugels})  yield that there is $n_2\ge n_1$ such that for $n\ge m\ge n_2$,
\begin{equation*}
(1-\delta)2^{-2\delta(n-m)}=(1-\delta)b^{-2\delta(n-m)\log_b2}\le \frac{\ell(b^n)}{\ell(b^m)}\le (1+\delta)b^{2\delta(n-m)\log_b2}=(1+\delta)2^{2\delta(n-m)}.
\end{equation*}
With \eqref{qk}, we obtain that for $n\ge n_2+m$,
\begin{align}
2^{-(q+\delta)m}\frac{\lambda^{-}_\delta }{\sqrt{1+\delta}}<\dfrac{ \beta_{n-m} }{s_n}<2^{-(q-\delta)m}\frac{\lambda^{+}_\delta }{\sqrt{1-\delta}}.\label{qk2}
\end{align}

Note that
\[
Z^{(n)}:=\lambda\sum_{m=1}^n 2^{-qm}Y_m\lra Z\qquad\text{in $L^p$.}
\]
By Minkowski's inequality and with $c:=\sup_m\|Y_m\|_p<\infty$,
\begin{align}
\bigg\|\frac{Z_n}{s_n}-Z^{(n)}\bigg\|_p&\le \sum_{m=1}^{n}\left|\frac{\beta_{n-m}}{s_n}-\lambda 2^{-qm}\right|\|Y_m\|_p\le c\sum_{m=1}^{n-n_2}\left|\frac{\beta_{n-m}}{s_n}-\lambda 2^{-qm}\right|+c\!\!\!\!\!\!\sum_{m=n-n_2+1}^n\left|\frac{\beta_{n-m}}{s_n}-\lambda 2^{-qm}\right|.\label{sum split eq}
\end{align}
To deal with the first sum on the right-hand side, we use  \eqref{lambdapm eq} and \eqref{qk2} to get that for $m\le n-n_2$,
\begin{align*}
\frac{\beta_{n-m}}{s_n}-\lambda 2^{-qm}&\le 2^{-(q-\delta)m} \bigg(\frac{\lambda_\delta^+}{\sqrt{1-\delta}}-\lambda\bigg)+\lambda(2^{-(q-\delta)m}-2^{-qm})\\
&\le \eps  2^{-(q-\delta)m}+\lambda 2^{-(q-\delta)m}(1-2^{-\delta m}).
\end{align*}
In the same way, we get the lower bound
$$ -\eps 2^{-(q+\delta)m}-\lambda 2^{-qm}(1-2^{-\delta m})\le \frac{\beta_{n-m}}{s_n}-\lambda 2^{-qm}.
$$
Using \eqref{1-2deltam eq} and \eqref{lambda times sum eq}, we thus get the following estimate for the first sum on the right-hand side of \eqref{sum split eq}
\begin{align*}
\sum_{m=1}^{n-n_2}\left|\frac{\beta_{n-m}}{s_n}-\lambda 2^{-qm}\right|\le \eps \sum_{m=1}^{n-n_2}2^{-(q-\delta)m}+\lambda \sum_{m=1}^{K}2^{-(q-\delta)m} \eps+\lambda \sum_{m=K+1}^{\infty}2^{-(q-\delta)m} \le \eps c',\end{align*}
 where $c':=2\sum_{m=0}^\infty 2^{-(q-\delta_0)m}+1$. For the right-most sum in \eqref{sum split eq}, we have 
\begin{equation}
\begin{split}
\sum_{m=n-n_2+1}^n\left|\frac{\beta_{n-m}}{s_n}-\lambda 2^{-qm}\right|&\le \sum_{m=n-n_2+1}^n\frac{\beta_{n-m}}{s_n}+\lambda\sum_{m=n-n_2+1}^n2^{-qm}\\
&\le\frac{n_2}{s_n}\max_{i=0,\dots,n_2-1}\beta_i+ \lambda 2^{- qn}\frac{2^{
   qn_2}-1}{2^q-1}.\label{n}
\end{split}
\end{equation}
We will show below that our assumptions imply that  $s_n\to\infty$. Therefore, the right-hand side of \eqref{n} can be made smaller than $\eps$ by choosing $n$ sufficiently large. Thus, we see that the $L^p$-distance between $Z_n/s_n$ and $Z^{(n)}$ is smaller than  $c(c'+1)\eps$ if $n$ is sufficiently large.
 
 Finally, we argue that $s_n\to\infty$ if \eqref{snL2n condition2} holds and $q>0$. To this end, let $g(x):=x^{2q}\ell(x)$, so that the denominator in \eqref{snL2n condition2} is equal to $g(2^n)$. Proposition 1.3.6 in \cite{BinghamGoldieTeugels} states that $g(x)\to\infty$, which in view of \eqref{snL2n condition2} yields our claim.\end{proof}

Now we briefly discuss the case where $q=0$ in \eqref{snL2n condition2}. In that case, we are typically in the regime of the central limit theorem. This is suggested, for instance, by \cite{ZhaoWoodroofeVolny}, where a central limit theorem for reversible Markov chains was obtained under the condition that the variance, $\sigma_n^2$, satisfies $\sigma^2_n=n\ell(n)$ for a slowly varying function $\ell$. Since $g(x):=x\ell(x)$ is regularly varying with index $\varrho=1$, we see that, in our situation,  this condition is a special case of condition \eqref{Thm_eq_1} with $q=0$. 
In the following proposition, we formulate  implications between stronger and weaker versions of  condition \eqref{snL2n condition2} with $q=0$.

Now we are going to formulate a simple central limit theorem  in the form  needed in the next section. We believe that this result is well known but were unable to find an exact reference to the literature. Due to the large number of existing variants of the central limit theorem, we  confine ourselves here to the basic case that will  actually be needed in \Cref{proofs section}.

\begin{lemma}\label{CLT prop}
Suppose that $Y_1,Y_2,\dots$ is an i.i.d.~sequence of symmetric $\{-1,+1\}$-valued Bernoulli random variables  
%and that condition {\rm(\ref{CLT cond prop e})}  in \Cref{CLT cond prop} holds, i.e.,
such that $s_{n-1}^2/s_n^2\to1$ as $n\ua\infty$.
 Then, for every $p\in[1,\infty)$,  the laws of $\frac1{s_n}Z_n$ converge in the $L^p$-Wasserstein metric to the standard normal distribution. \end{lemma}

\noindent\emph{Proof.}  As $s_{n-1}^2/s_n^2\to1$, it then follows that
\begin{align*}
\lim_{n\ua\infty}\frac{\beta_n^2}{s^2_n}&=\lim_{n\ua\infty}\Big(1-\frac{s_{n-1}^2}{s^2_n}\Big)=0.
\end{align*}
It is hence easy to see that the sequence  $ \{Z_n\}_{n \in \bN_0}$ satisfies Lindeberg's condition, which yields that $\frac1{s_n}Z_n$ converges in law to $N(0,1)$. Thus, to conclude convergence in the $L^p$-Wasserstein metric, it is sufficient to show that $\bE[|s_n^{-1}Z_n|^q]$ is uniformly bounded in $n$ for some $q>p$; this follows, e.g.,  from  Lemma 5.61 and Corollary A.50 in \cite{FoellmerSchiedBuch}.
We actually show a stronger condition, namely  that $\bE[e^{|\lambda Z_n/s_n|}]$ is uniformly bounded for all $\lambda\in\bR$. Since $e^{|x|}\le e^x+e^{-x}$, we can drop the absolute value in the exponent. Then,
\begin{align*}
\log \bE[e^{\lambda Z_n/s_n}]=\log\prod_{k=0}^n\cosh(\lambda\beta_k/s_n)=\sum_{k=0}^n\log\cosh(\lambda\beta_k/s_n)\le \sum_{k=0}^n\frac{\lambda^2\beta_k^2}{2s_n^2}=\frac{\lambda^2}2.~~~~~~~~~~~~~~{\qedsymbol}
\end{align*}

\section{$\Phi$-variation of functions in the Takagi class}\label{Phi variation section}

In this section, we will apply the limit theorems from \Cref{limit section}
 to the problem of computing the $\Phi$-variation of functions in the Takagi class. To this end,
let $\Phi: [0,\infty)\to[0,\infty)$ be a  function with $\Phi(0)=0$. The \emph{$\Phi$-variation of a continuous function $f:[0,1]\to\bR$ along the sequence of dyadic partitions} is defined as 
$$\<f\>^\Phi_t:=\lim_{n\uparrow\infty}\sum_{k=0}^{\lfloor t2^n\rfloor}\Phi\big(|f((k+1)2^{-n})-f(k2^{-n})|\big),\qquad 0\le t\le 1,
$$
provided that  the limit exists for all $t\in[0,1]$. Due to the uniform continuity of $f$, we may, without loss of generality, restrict $\Phi$ to the domain $[0,1)$. For $\Phi(x)=x$ and $t=1$, we obtain the total variation of $f$. For $\Phi(x)=x^2$, we obtain the usual quadratic variation, and for $\Phi(x)=x^p$ with $p\ge1$ the $p^{\text{th}}$ variation. Here, we are interested in functions for which the correct $\Phi$ may   \emph{not} be of the form $\Phi(x)=x^p$ but rather has a more complex structure. 

Let $\varphi(t)=\min_{z\in\bZ}|z-t|$ denote the tent map. The Takagi class, as introduced by Hata and Yamaguti \cite{HataYamaguti}, consists of all  functions $f:[0,1]\to \bR$ that admit an expansion of the form 
\begin{align}\label{vdw}
f(t):=\sum_{m=0}^\infty \alpha_m\varphi(2^mt),\qquad t\in[0,1],
\end{align}
where   $\{\alpha_m\}_{m \in \bN_0}$  is a sequence of real numbers for which the series $\sum_{m=0}^\infty\alpha_m$ converges absolutely.  Clearly, the series on the right-hand side of \eqref{vdw} converges  uniformly in $t\in[0,1]$, so that $f$ is indeed a well defined continuous function. Typically, the functions in the Takagi class have a fractal structure. For instance, the choice $\alpha_m=2^{-m}$ corresponds to the classical, nowhere differentiable Takagi function, which was originally introduced in \cite{Takagi} but rediscovered many times. More generally, the choice $\alpha_m=a^m$ for some $a\in(-1,1)$ yields the class of so-called Takagi--Landsberg functions.  Analytical properties of the functions in the Takagi class, such as differentiability, are discussed in \cite{Kono}.

In our first result, we take  $f$ as in \eqref{vdw} as given and relate the asymptotic behavior of 
\begin{equation*}
  V_n^{(p)}:= \sum_{k = 0}^{ 2^n-1}|f((k+1) 2^{-n}) - f(k 2^{-n})|^p,\qquad p\ge1, \  n\in\bN,
\end{equation*}
to the one of
\begin{equation}\label{eq sn}
    s_n^2:=\sum_{m=0}^{n-1}\alpha_m^24^m,\qquad n\in\bN.
\end{equation}
Note that $V_n^{(1)}$ converges to the total variation of $f$,  defined as usual through taking the supremum over \emph{all} possible partitions of $[0,1]$ (see, e.g.,  Theorem 2 in \S5 of Chapter VIII in~\cite{Natanson}).  The following proposition provides a priori bounds on the behavior of $V_n^{(p)}$, based on the range of possible growth rates of $s_n$.

\begin{proposition}\label{Prop Khine}
Let 
\begin{equation*}
   q_* := \liminf_{n \ua \infty}\frac{1}{n}\log_2 s_n\qquad\text{and}\qquad  q^* := \limsup_{n \ua \infty}\frac{1}{n}\log_2 s_n.
\end{equation*}
Then, for $p\ge 1$, we have $\liminf_nV^{(p)}_n=0$ if $p(1-q_*)>1$ and  $\limsup_nV_n^{(p)}=\infty$  if $p(1-q^*)<1$.
\end{proposition}

\begin{example}Faber \cite{Faber07} considered the  function $f$ with 
$$\alpha_m=\begin{cases}\frac1{10^n}&\text{if $m=n!$ for some $n\in\bN$,}\\
0&\text{otherwise,}
\end{cases}
$$
and showed that it is nowhere differentiable. It is easy to see that this choice leads to $q^*\ge1$ and $q_*\le0$, which implies that $\liminf_nV^{(p)}_n=0$ and $\limsup_nV^{(p)}_n=\infty$ for all $p\ge1$. 
\end{example}

A more interesting case is the situation in which $f$ is such that $q:=\lim_n\frac1n\log_2s_n$ exists and satisfies $0\le q<1$. Then \Cref{Prop Khine} implies that $p:=1/(1-q)$ is the critical exponent for the power variation of $f$. The standard paradigm in the literature is to expect that $f$ then admits a finite $p^{\text{th}}$ variation. This is true for the typical sample paths of a continuous semimartingale with $p=2$ and for the trajectories of fractional Brownian motion with Hurst\label{Hurst} exponent $H=1/p$. As shown in \cite{MishuraSchied2,SchiedZZhang}, it is also true for the Takagi--Landsberg functions.   In the sequel, we shall in particular investigate the situation in which $\lim_nV_n^{(p)}$ exists but is either zero or infinity. In this case, $p^{\text{th}}$ power variation is clearly no longer sufficient for characterizing the exact \lq\lq roughness" of $f$.

 To state our first main result,  we recall that Corollary 1.4.2 in \cite{BinghamGoldieTeugels} allows us to assume without loss of generality that all regularly varying functions  are  bounded on compact intervals (see the footnote on page~\pageref{reg varying footnote} for the definition of a regularly varying function). If $g$ is  regularly varying  and $q\in[0,1)$, consider  
\begin{equation}\label{Phiq eq}
    \Phi_q(x)=\Phi_{q,g}(x):= x^{\frac{1}{1-q}}g\Big(\dfrac{-\log_2 x}{1-q}\Big)^{-\frac{1}{2(1-q)}}
\end{equation}
for $0<x<1$.
By Proposition 1.5.7  in \cite{BinghamGoldieTeugels}, $y\mapsto g((\log_2y)/(1-q))$ is slowly varying, and so $\psi_q(y):=\Phi_q(1/y)$ is regularly varying with index $-1/(1-q)$. We conclude that $\Phi_q(x)\to0$ as $x\da0$. It therefore makes sense to extend $\Phi_q$ to the domain $[0,1)$ by setting $\Phi_q(0):=0$. Note also that $\Phi_q$ is \emph{slowly varying at zero} with index $1/(1-q)$ and that the functions arising in this manner form a very wide class of functions that are regularly varying at zero.

\begin{theorem}\label{Theorem on Phi-variation}
For $f$ as in \eqref{vdw}, let $s_n^2$ be as in \eqref{eq sn}
and suppose that there exist $q\in [0,1)$ and a regularly varying  function $g$ such that \begin{equation}\label{Thm_eq_1}
    \lim_{n \ua \infty}\dfrac{s_n^2}{2^{2qn}g(n)} = 1.
\end{equation}
Let moreover  $\{Y_m\}_{m \in \bN}$ be an  i.i.d.~sequence of symmetric $\{-1,+1\}$-valued Bernoulli random variables  and $\Phi_q$ as in \eqref{Phiq eq}.
Then the following hold:
\begin{enumerate}
\item The function $f$ is of bounded variation if and only if $\lim_ns_n^2<\infty$ and, in this case, the total variation of $f$ is equal to 
$\bE[|\wt Z|]$, where $\wt Z:=\sum_{m=0}^\infty \alpha_m2^mY_{m+1}$. 
\item If $s_n^2\to\infty$ and $q=0$, then the $\Phi_0$-variation of $f$ is given by
\begin{equation}\label{Phi0 variation eq}
    \<f\>^{\Phi_0}_t = \sqrt{\frac{2}{\pi}}\cdot t.
\end{equation}
\item If $q>0$, then  the $\Phi_q$-variation of $f$ is given by 
\begin{equation}\label{Bernoulli thm c eqn}
    \<f\>_t^{\Phi_q} = \bE[|Z|^{\frac{1}{1-q}}]\cdot t\qquad\text{where}\qquad Z:=\sqrt{2^{2q}-1}\sum_{m=1}^\infty 2^{-qm}Y_m.
\end{equation}
\end{enumerate}

\end{theorem}

 \begin{remark}The random variable $Z$ in \eqref{Bernoulli thm c eqn} can be represented as 
 $$Z=\sqrt{1-2^{-2q}}\sum_{m=0}^\infty 2^{-qm}Y_{m+1}
 $$
 and hence has as its law a scaled version of the infinite Bernoulli convolution with parameter $2^{-q}$. These laws exhibit some fascinating properties and have been studied in their own right for many decades; see, e.g.,~\cite{PeresSchlagSolomyak} and the references therein.
 \end{remark}
 
In the next two results, we consider again the situation of \Cref{Prop Khine} and take a closer look at what happens to the power variation of a function $f$ in the Takagi class satisfying \eqref{Thm_eq_1}.
 
 \begin{corollary}\label{g limit cor} Suppose that \eqref{Thm_eq_1}
 holds with $q\in[0,1)$ and let $p:=1/(1-q)$.
 \begin{enumerate}
  \item The function $f$ has infinite $r^{\text{th}}$ variation for $1\le r<p$ and vanishing $r^{\text{th}}$  variation for $r> p$.
 \item  If $\lim_ng(n)=c$ for some number $c\in[0,\infty]$, then the $p^{\text{th}}$ variation of $f$ satisfies $\<f\>^{(p)}_t=c^{p/2}\<f\>^{\Phi_q}_t$ for $t\in[0,1]$.
 \end{enumerate}
 \end{corollary}

If $g$ is regularly varying with index $\varrho\in\bR$, then $g(x)\to0$ for $\varrho<0$ and $g(x)\to\infty$ for $\varrho>0$ according to Proposition B.1.9 (1) in \cite{deHaanFerreira}. If $\varrho=0$, then $g$ is slowly varying and may or may not converge to a number  $c\in [0,\infty]$.
 In conjunction with \Cref{g limit cor}, we thus get immediately the following corollary.

\begin{corollary}\label{rho corollary} Suppose that \eqref{Thm_eq_1}
 holds with $q\in[0,1)$ and a function  $g$ that is regularly varying with index $\varrho\in\bR$.   We let $p:=1/(1-q)$.
 \begin{enumerate}
 \item If $\varrho<0$, then $f$ has vanishing $p^{\text{th}}$ variation.
 \item  If $\varrho>0$, then $f$ has infinite $p^{\text{th}}$ variation.
 \end{enumerate}
\end{corollary}

\begin{example}\label{Takagi example}
Let us consider the classical Takagi function, 
\begin{equation*}
    f(t) = \sum_{m = 0}^{\infty} 2^{-m}\varphi(2^mt),
\end{equation*}
as discussed in the introduction. Then $s_n^2=n$ and so we can take $q=0$ and $g(x):=x$. Corollary \ref{rho corollary}  thus immediately yields the well-known fact that $f$ is of unbounded total variation. Moreover, we have 
$
    \Phi_0(t) = {x}/{\sqrt{-\log_2 x}}$, and so we recover here a special case of Theorem 1.2 in \cite{HanSchiedZhang1}.
\end{example}

Allaart \cite{Allaartflexible} introduced an extension of the Takagi class by multiplying each replicate of the tent map in the expansion \eqref{vdw}
 with an arbitrary sign. More precisely, he considered the class of functions of the form 
 \begin{equation}\label{flexible class}
 f(t):=\sum_{m=0}^\infty \alpha_m \sigma_m(t)\varphi(2^mt),\qquad t\in[0,1],
 \end{equation}
 where $\sigma_m:[0,1]\to\{-1,+1\}$ is  constant on each interval $[(k-1)2^{-m},k2^{-m})$ for $k=1,\dots,2^m$,  and $\{\alpha_m\}_{m \in \bN_0}$  is a sequence of real numbers for which the series  $\sum_{m}\alpha_m$ converges absolutely. Then $f$ is well-defined and continuous due to the uniform convergence of the series in \eqref{flexible class}. Certain deterministic choices for $\sigma_m$ yield fractal functions studied in other fields such as information theory; see \cite{Allaartflexible,AllaartKawamura} and the references therein. When the signs of $\sigma_m$ are chosen in a random manner, the expansion \eqref{flexible class} becomes closely related to  the Wiener--L\'evy--Ciesielski expansion of Brownian motion by means of the Faber--Schauder functions, and the functions \eqref{flexible class} with fixed coefficients  $\{\alpha_m\}_{m \in \bN_0}$ but random signs form a non-Gaussian stochastic process with rough sample paths.  More precisely,  for each $m$, let  $\{\xi_{m,k}\}_{k = 1,\cdots,2^m}$ be a sequence of $\{-1,+1\}$-valued Bernoulli random variables and define 
 \begin{equation}\label{stoch process}
 \sigma_m(t):=\sum_{k=1}^{2^m}\xi_{m,k}\Ind{[(k-1)2^{-m},k2^{-m})}(t)\qquad\text{and}\qquad X_t:=\sum_{m=0}^\infty \alpha_m \sigma_m(t)\varphi(2^mt).
 \end{equation}
 Then the following corollary states that $\{X_t\}_{t\in[0,1]}$ is a stochastic process whose sample paths have the deterministic linear $\Phi_q$-variation \eqref{Bernoulli thm c eqn} if the sequence  $\{\alpha_m\}_{m \in \bN_0}$ satisfies the conditions on \Cref{Theorem on Phi-variation}.
 An illustration of this stochastic process is provided in \Cref{rho figure}. For the choice $\alpha_m=a^m$ with $a\in(-1,1)$, the $p^{\text{th}}$ variation of the functions \eqref{flexible class} was studied in \cite{MishuraSchied2,SchiedTakagi}.
 
 \begin{corollary}\label{flexible Cor}The statements of \Cref{Theorem on Phi-variation} and of the Corollaries \ref{g limit cor} and \ref{rho corollary} remain fully valid for the functions of the form \eqref{flexible class}.
 \end{corollary}

When the goal is to utilize \Cref{Theorem on Phi-variation}
 for the construction of functions with prescribed $\Phi_q$-variation for a given function $g$, it can be inconvenient that the condition \eqref{Thm_eq_1} is formulated in terms of the sums $s_n^2=\sum_{m=0}^{n-1}\alpha_m^24^m$ rather than in terms of the coefficients $\alpha_m$ themselves. To deal with this issue, we are going to formulate \Cref{CLT cond prop}, which will also be used in the proofs of our main results.

\begin{remark}\label{ell g remark} Let us recall the following facts on slowly and regularly varying functions so as to put the statement of  the following \Cref{CLT cond prop}  into the context of \Cref{Theorem on Phi-variation}. \begin{enumerate}
\item If $g$ is a regularly varying function and $b>1$, then $\ell(x):=g(\log_bx)$ is slowly varying according to Proposition 1.5.7 (ii)  in \cite{BinghamGoldieTeugels}. Moreover, our condition \eqref{Thm_eq_1} implies that 
$$\lim_{n\to\infty}\frac{s_n^2}{2^{2qn}\ell(b^n)}=1.
$$
\item Recalling that we assume all regularly (and hence slowly) varying functions to be locally bounded, and thus locally integrable, we may define
\begin{equation}\label{ell def eq}
\ell(x):= \frac1{\log b}\int_1^x\frac1tL(t)\,dt=\int_0^{\log_bx}L(b^s)\,ds
\end{equation}
for a given slowly varying function $L$.
Then, according to Proposition 1.5.9 (a) in \cite{BinghamGoldieTeugels}, the function $\ell$ is slowly varying  and satisfies $\ell(x)/L(x)\to+\infty$ as $x\ua\infty$.
\end{enumerate}
\end{remark}

\begin{proposition}\label{CLT cond prop} 
Let  $\{\beta_m\}_{m \in \bN_0}$ be a sequence of real numbers, $s_n^2=\sum_{m=0}^{n-1}\beta_m^2$, and $b>1$. 
\begin{enumerate}[{\rm (a)}]
\item\label{CLT cond prop (a)} Suppose that $s_n\to\infty$ as $n\ua\infty$, let $L$ be a slowly varying function, define $\ell$ as in \eqref{ell def eq}, and consider the following two conditions:
\begin{enumerate}[{\rm (i)}]
\item\label{CLT cond prop (a) (i)} $\beta_n^2/L(b^n)\to1$ as $n\ua\infty$.
\item\label{CLT cond prop (a) (ii)} 
 $s_n^2/\ell(b^n)\to1$ as $n\ua\infty$.
\end{enumerate}
Then {\rm(i)}$\Rightarrow${\rm(ii)}.

\item If $\ell$ is a slowly varying function and $q>0$, the following are equivalent:
\begin{enumerate}[{\rm (i)}]\addtocounter{enumii}{2}
\item $\beta_n^2/2^{2qn}\ell(b^n) \to1$ as $n\ua\infty$.
\item $s_n^2/2^{2qn}\ell(b^n) \to(2^{2q}-1)^{-1}$ as $n\ua\infty$.
\end{enumerate}
\item\label{CLT cond prop e} Any of the conditions  {\rm(i)}--{\rm(iv)} implies $s_{n-1}^2/s_n^2\to 2^{2q}$ as $n\ua\infty$, where we take $q=0$ in the context of~{\rm({a})}.
\end{enumerate}
\end{proposition}
 
 The proofs of \Cref{Prop Khine}, \Cref{Theorem on Phi-variation}, \Cref{flexible Cor}, \Cref{g limit cor}, and \Cref{CLT cond prop}  are given in \Cref{proofs section}.  \Cref{Theorem on Phi-variation} and \Cref{CLT cond prop} yield immediately the following result, which provides a simple construction for a function with prescribed $\Phi_q$-variation. When applying it to the stochastic process \eqref{stoch process}, it provides a straightforward way of constructing a stochastic process whose sample paths have deterministic and linear $\Phi_q$-variation; see \Cref{rho figure} for an illustration.
\goodbreak
 
 \begin{corollary} Let $g$ be a regularly varying function,  $q\in[0,1)$, and $\Phi_q$ as in \eqref{Phiq eq}. 
\begin{enumerate}
\item For $q\in(0,1)$,  the functions $f$ in \eqref{vdw}
and \eqref{flexible class} with coefficients  
$$\alpha_m:=2^{-m(1-q)}\sqrt{(2^{2 q} - 1)g(m)} $$
 have linear $\Phi_q$-variation \eqref{Bernoulli thm c eqn}.
\item For $q=0$, we assume in addition that $g(x)\to\infty$ as $x\ua\infty$ and that $g$ is absolutely continuous with regularly varying derivative $g'$. Then the functions $f$ in \eqref{vdw}
and \eqref{flexible class} with coefficients 
$$\alpha_m:=2^{-m}\sqrt{g'(m)}$$
 have linear $\Phi_q$-variation \eqref{Phi0 variation eq}.
\end{enumerate} 
   \end{corollary}

\begin{proof} (b) Since the function $L(x):=g'(\log_2x)$ is slowly varying by \Cref{ell g remark}, \Cref{CLT cond prop} (a) yields that $s_n^2/\ell(2^n)\to1$, where $\ell(x)=\int_0^{\log_2x}L(2^t)\,dt=g(\log_2x)-g(0)$. Hence, the conditions of  \Cref{Theorem on Phi-variation} (b) are satisfied. 
Part (a) is now obvious. \end{proof}

 \begin{figure}[h]
 \centering
\includegraphics[width=5.75cm]{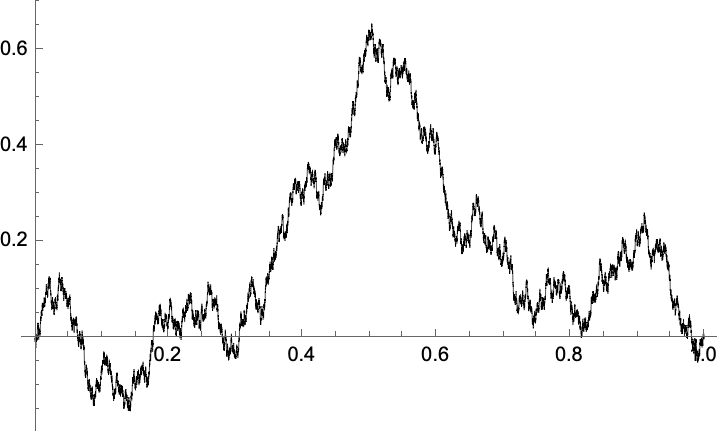}
\quad \includegraphics[width=5.75cm]{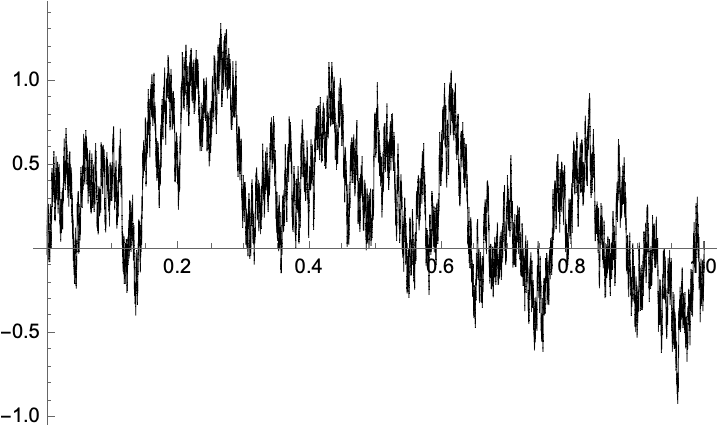}
\quad 
\includegraphics[width=5.75cm]{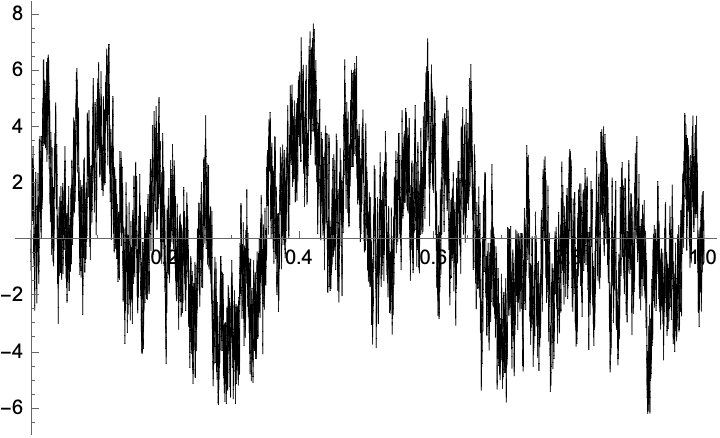}
\\
 \medskip
\includegraphics[width=5.75cm]{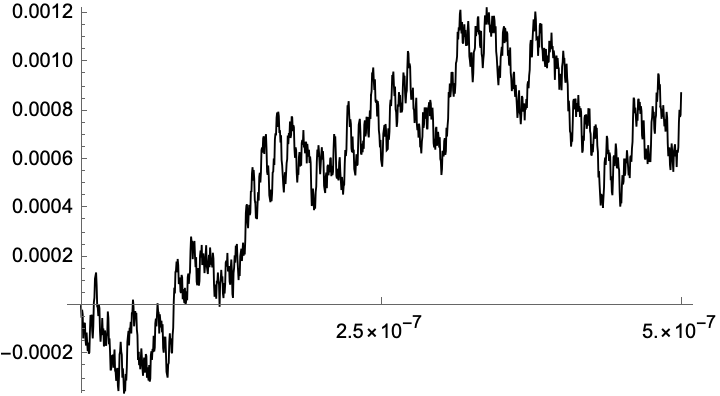}
\quad \includegraphics[width=5.75cm]{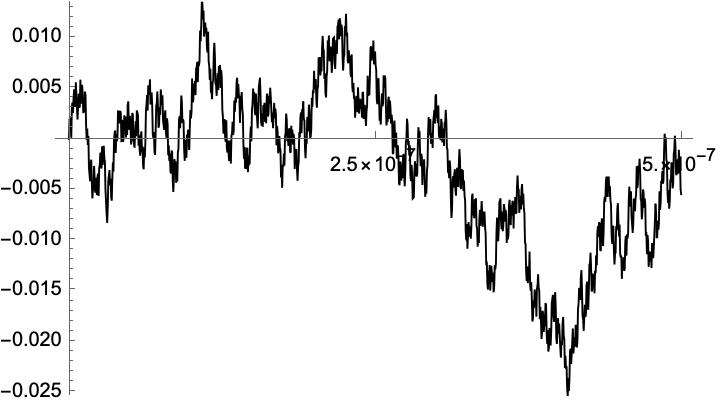}
\quad 
\includegraphics[width=5.75cm]{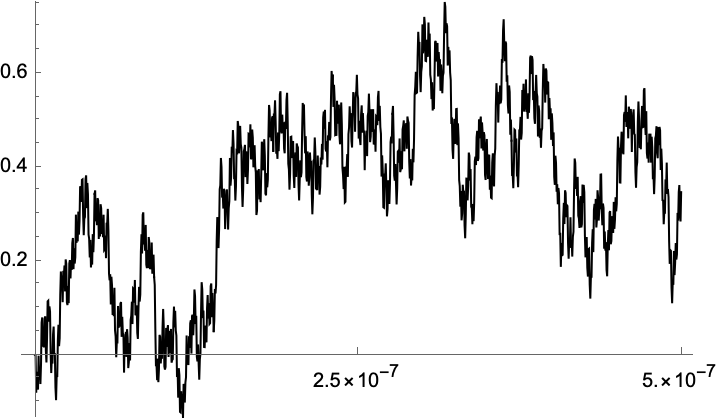}
\caption{Sample paths  of the stochastic process $\{X_t\}_{t\in[0,1]}$ defined in \eqref{stoch process} with 
$\alpha_m=2^{-m(1-q)}\sqrt{(2^{2 q} - 1)g(m)} $  for $q=0.7$, $g(x)=(1+x)^\varrho$
with $\rho=-2$ (left), $\rho=0$ (center), $\rho=2$ (right), and signs  $\{\xi_{m,k}\}_{m \in \bN_0, k = 1,\cdots,2^m}$ that are  i.i.d.~$\{-1,+1\}$-valued Bernoulli random variables. For each $\varrho$, a new realization of  $\{\xi_{m,k}\}$ was chosen. One can see from the three top panels that the parameter $\varrho$ has a strong influence on the macroscopic fluctuations of the trajectories. 
The fact that all three   functions share the same critical exponent $p=1/(1-q)=10/3$ as defined in \Cref{Prop Khine}  suggest a common Hurst exponent of $H=1/p=0.3$ (see the discussion on page~\pageref{Hurst}). One would thus expect the fluctuations of the three functions to be  similar on a microscopic scale. This effect can indeed be observed in the lower  panels below, which show the same three functions after zooming into the interval $[0,\eps]$ for $\eps=5\cdot10^{-7}$. 
 }
\label{rho figure}
 \end{figure}

\section{Proofs}\label{proofs section}

\begin{proof}[Proof of \Cref{CLT cond prop}] (i)$\Rightarrow$(ii): For any given $\eps > 0$, there exists $n_0$ such that for all $m > n_0$, we have 
$
     L(b^m)(1 - \eps) <\beta_m^2 <  L(b^m)(1 + \eps).
$
Hence, for all $\nu \ge n_0$ and $n>\nu $,
\begin{equation}\label{thm_normal_eq_2}
    \sum_{m = 0}^{\nu }\beta_m^2 + (1-\eps)\sum_{m = \nu +1}^{n-1} L(b^m) < \sum_{m = 0}^{n-1}\beta_m^2 <  \sum_{m = 0}^{\nu }\beta_m^2 +(1+\eps)  \sum_{m = \nu +1}^{n-1} L(b^m).
\end{equation}
The Potter bounds  (e.g., Theorem 1.5.6 in \cite{BinghamGoldieTeugels}) yield for   every $\delta>0$ some $n_1\in\bN$ such that 
\begin{equation*}\label{Potter bounds eq}
(1-\eps)b^{\delta(m-t)}\le \frac{L(b^t)}{L(b^m)}\le (1+\eps)b^{\delta(t-m)}\qquad\text{for all $t,m\ge n_1$.}
\end{equation*}
We hence get that for $m\ge n_1$,
\begin{align*}
L(b^m) (1-\eps)\frac{1-b^{-\delta}}{ \delta\log b}
\le \int_m^{m+1} L(b^t)\,dt\le L(b^m) (1+\eps)\frac{b^{\delta}-1}{ \delta\log b}.
\end{align*}
Since $\pm(1-b^{\mp\delta})/(\delta\log b)\to 1$ as $\delta\da0$, we can choose $\delta$ small enough  such that for any $m\geq n_1$,
\begin{equation*}
(1-\eps)^2 L(b^m) 
\le \int_m^{m+1} L(b^t)\,dt\le (1+\eps)^2 L(b^m).
\end{equation*}
Using this estimate in \eqref{thm_normal_eq_2}  and dividing both sides of the result by $\int_{0}^{n}  L(b^t)\,dt=\ell(b^n)$  gives
\begin{equation}\label{series int comparison eq}
     \frac{\sum_{m = 0}^{n_1-1}\beta_m^2}{\int_{0}^{n} L(b^t)\,dt} + \frac{(1-\eps)^3 \int_{n_1}^{n}  L(b^t)\,dt}{\int_{0}^{n} L(b^t)\,dt}< \frac{s_n^2}{\ell(b^n)}<  \frac{\sum_{m = 0}^{n_1-1}\beta_m^2}{\int_{0}^{n} L(b^t)\,dt} + \frac{(1+\eps)^3\int_{n_1}^{n}  L(b^t)\,dt}{\int_{0}^{n} L(b^t)\,dt}.
\end{equation}
Since $s_n\to\infty$, the inequalities in \eqref{series int comparison eq} imply that we must also have $\ell(b^n)=\int_{0}^{n}  L(b^t)\,dt\to\infty$ as $n\ua\infty$. It follows that the left- and right-hand sides of \eqref{series int comparison eq} tend to $(1-\eps)^3$ and $(1+\eps)^3$ as $n\ua\infty$, respectively. Sending $\eps\da0$ now yields (ii).

(iv)$\Leftrightarrow$(iii): It follows from the Stolz theorem and its converse given in Lemma 3.1 of \cite{MishuraSchied}  that the existence of one of the following limits entails the existence of the other one and that in this case both must be equal: 
\begin{align*}
\lim_{n^\ua\infty}\frac{s_n^2}{2^{2qn}\ell(b^n)}\qquad\text{and}\qquad \lim_{n^\ua\infty}\frac{s_{n+1}^2-s^2_{n}}{2^{2q(n+1)}\ell(b^{n+1})-2^{2qn}\ell(b^{n})}.
\end{align*}
Moreover, if it exists, the limit on the right-hand side is equal to
$$\lim_{n^\ua\infty}\frac{\beta_n^2}{2^{2qn}\ell(b^n)}\cdot\frac1{ (2^{2q}\frac{\ell(b\cdot b^{n})}{\ell(b^n)}-1)},
$$
and here   the second factor  converges to $(2^{2q}-1)^{-1}$.

(c): We have
 $$\lim_{n\ua\infty}\frac{s_{n-1}^2}{s_n^2}=2^{2q}\lim_{n\ua\infty}\frac{\ell(b^{n-1})}{\ell(b\cdot b^{n-1})}=2^{2q}.
 $$
 This concludes the proof.
 \end{proof}

Now we turn toward proving \Cref{Theorem on Phi-variation} (and also \Cref{flexible Cor} along the way).
Following \cite{HanSchiedZhang1,SchiedZZhang}, we let $(\Omega,\cF,\bP)$ be a probability space supporting an independent sequence $U_1,U_2, \cdots$ of symmetric $\{0,1\}$-valued Bernoulli random variables. Then we define the stochastic process $R_m := \sum_{k = 1}^{m}U_k2^{k-1}$ and set
\begin{equation}\label{YmZm eq}
Y_m:=\frac{\varphi((R_m+1) 2^{-m})-\varphi(R_m 2^{-m})}{ 2^{-m}},\qquad \beta_m:= 2^m\alpha_m,\qquad\text{and}\qquad s_n^2:=\sum_{m=0}^{n-1}\beta_m^2.
\end{equation}
Proposition 3.2 (a) in \cite{SchiedZZhang} states that $Y_1,Y_2,\dots$ is an i.i.d.~sequence of symmetric $\{-1,+1\}$-valued Bernoulli random variables. Moreover, the proof of Theorem 2.1 in \cite{MishuraSchied2} gives the same result in the context of \Cref{flexible Cor}, so that all subsequent observations are also valid for the  functions of the form \eqref{flexible class}.

 Note that $R_m$ has a uniform distribution on $\{0,\dots, 2^m-1\}$. 
Therefore, for $n\in\bN$ such that all increments $\big|f((k+1)2^{-n})-f(k2^{-n})\big|$ are less than 1,
\begin{equation*}\label{Vn def eq}
V_n:=\sum_{k=0}^{2^n-1}\Phi\big(\big|f((k+1)2^{-n})-f(k2^{-n})\big|\big)=2^n\mathbb E\Big[\Phi\big(\big|f((R_n+1)2^{-n})-f(R_n2^{-n})\big|\big)\Big].
\end{equation*}
To analyze the expectation on the right, let the $n^{\text{th}}$ truncation of $f$ be given by
$
f_n(t)=\sum_{m=0}^{n-1} \alpha_m\varphi(2^mt).
$
With the notation introduced in \eqref{YmZm eq}, we get
\begin{align*}
f((R_n+1)2^{-n})-f(R_n2^{-n})&=f_n((R_n+1)2^{-n})-f_n(R_n2^{-n})\\
&=2^{-n} \sum_{m=0}^{n-1}\beta_m\frac{\varphi((R_n+1)2^{m-n})-\varphi(R_n2^{m-n})}{2^{m-n}}\\
&=2^{-n} \sum_{m=0}^{n-1}\beta_mY_{n-m},
\end{align*}
where in the last step we have used that $\varphi(x+R_n2^{-m})=\varphi(x+R_m2^{-m})$ due to the periodicity of $\varphi$.
Hence, if we define 
$$Z_n:=\sum_{m=0}^{n-1}\beta_mY_{n-m}=\sum_{m=1}^{n}\beta_{n-m}Y_{m},
$$
then
\begin{equation}\label{eqV}
    V_n=2^n\bE\big[\Phi(| 2^{-n}Z_n|)\big].\end{equation}
Thus, the $\Phi$-variation of $f$ is determined by  the limit of the right-hand expectations as $n\ua\infty$.

\begin{proof}[Proof of \Cref{Prop Khine}] Taking $\Phi(x)=x^p$, we get from \eqref{eqV} that $V_n^{(p)}=2^{n(1-p)}\bE[|Z_n|^p]$. By Khintchine's inequality, there exist constants $A_p$ and $B_p$ only depending on $p$ such that $0 < A_p \le B_p < \infty$, and 
\begin{equation*}
   A_p2^{n(1-p+pq_n)} =  A_p2^{n(1-p)}(s_n^2)^{p/2} \le V^{(p)}_n \le B_p 2^{n(1-p)}(s_n^2)^{p/2} = B_p2^{n(1-p+pq_n)},
\end{equation*}
where  $q_n:=\frac1{n}\log_2 s_n$.  From here, the assertion is straightforward.\end{proof}

Now we prove \Cref{Theorem on Phi-variation}.
For better accessibility, we have divided the proof  into several parts.

\begin{proof}[Proof of \Cref{Theorem on Phi-variation} {\rm(a)}]
 Assuming that $\lim_ns_n^2<\infty$, we define $\beta_m=\alpha_m2^m$ and
 \begin{equation*}\label{tilde Z eq}
 \wt Z_n:=\sum_{m=0}^{n-1}\beta_mY_{m}.
 \end{equation*}
 Then the exchangeability of the sequence $\{Y_m\}_{m\in\bN_0}$ implies that $Z_n$ and $\wt Z_n$ have the same law. Clearly, $\wt Z_n\to \wt Z:=\sum_{m=0}^\infty\beta_mY_{m}$ in $L^1$. Therefore, \eqref{eqV} yields for the choice $\Phi(x):=x$,
$$\sum_{k=0}^{2^n-1}|f((k+1)2^{-n})-f(k2^{-n})|=2^n\bE[|2^{-n}Z_n|]\lra\bE[|\wt Z|].
$$
Since $f$ is continuous, this limit must coincide with the total variation of $f$ (see, e.g., Theorem 2 in \S5 of Chapter VIII in~\cite{Natanson}). Conversely, parts (b) and (c) of \Cref{Theorem on Phi-variation} will imply that $f$ can only be of finite total variation if $\lim_ns_n^2<\infty$.
\end{proof} 

\begin{proof}[Proof of \Cref{Theorem on Phi-variation} {\rm(b)} for $t=1$] Here, we identify the $\Phi_0$-variation $\<f\>_1^{\Phi_0}$ at time $1$. The linearity of the ${\Phi_0}$-variation $t\mapsto \<f\>_t^{\Phi_0}$ will be proved subsequently.  Condition \eqref{Thm_eq_1} for $q=0$ together with our assumption $s_n^2\to\infty$ implies via Proposition 1.5.1  in \cite{BinghamGoldieTeugels} that $g$ is regularly varying with index $\varrho\ge0$. Moreover, 
\Cref{ell g remark}, \Cref{CLT prop}, and \Cref{CLT cond prop} yield that the law and the first moment of $Z_n/s_n$ converge, respectively, to $N(0,1)$ and to its first moment.

Let us write $\psi(x):=\sqrt{g(x)}$ so  that $2^n{\Phi_0}(2^{-n}|Z_n|)=|Z_n|/\psi(n-\log_2|Z_n|)$. 
 It follows that
   \begin{align}\label{Z psi eq}
2^n{\Phi_0}(2^{-n}|Z_n|)\Ind{\{|Z_n|/\psi(n)\ge\eps\}}=  \frac{\psi(n)}{\psi(n-\log_2|Z_n|)}\cdot \frac{ |Z_n|}{\psi(n)}\Ind{\{|Z_n|/\psi(n)\ge\eps\}}.\end{align}
 To deal with the first factor on the right, we  estimate $n-\log_2|Z_n|$. First, on $\{|Z_n|/\psi(n)\ge\eps\}$, we have $\log_2|Z_n|\ge \log_2\eps+\ell(n)$, where $\ell(x):=\log_2\psi(x)$ is slowly varying by  Proposition 1.5.7  in \cite{BinghamGoldieTeugels} and hence satisfies $\ell(n)/n\to0$ according to Proposition 1.3.6 (v) in \cite{BinghamGoldieTeugels}.   Therefore, there exists $n_3\in\bN$ such that
 \begin{equation}\label{q=0 n- log2|Zn| upper bound}
 n-\log_2|Z_n|\le n-\log_2\eps-\ell(n)\le(1+\eps) n\qquad\text{on $\{|Z_n|/\psi(n)\ge\eps\}$ for all $n\ge n_3$.}
 \end{equation}
 Second, to get a lower bound, we use that $|Z_n|\le\sum_{m=0}^{n-1}|\beta_m|$. Jensen's inequality gives furthermore that 
 \begin{equation}\label{|Zn| eq}
 |Z_n|\le\sum_{m=0}^{n-1}|\beta_m| \le n\cdot \sqrt{\frac1n\sum_{m=0}^{n-1}\beta_m^2} =\sqrt ns_n.
 \end{equation}
 By \eqref{Thm_eq_1}, there is $n_4\in\bN$ such that $s_n\le(1+\eps)\psi(n)$ for all $n\ge n_4$. Thus, there is $n_5\ge n_4$ such that 
\begin{equation}\label{n-log2|Zn| lower bound}
n-\log_2|Z_n|\ge n-\log_2(1+\eps)-\frac12\log_2n-\ell(n)\ge (1-\eps)n\qquad\text{for all $n\ge n_5$.}
\end{equation}
Combining \eqref{q=0 n- log2|Zn| upper bound}
 and \eqref{n-log2|Zn| lower bound} now yields
\begin{equation}\label{psi/psi(log..) lower bound}
\frac{\psi(n)}{\psi(n-\log_2|Z_n|)}\ge \inf_{1-\eps\le\lambda\le 1+\eps}\frac{\psi(n)}{\psi(\lambda n)}
\qquad\text{for $n\ge n_5$ on $\{|Z_n|/\psi(n)\ge\eps\}$.}
\end{equation}
According to the uniform convergence theorem for regularly varying functions (Theorem 1.5.2 in \cite{BinghamGoldieTeugels}) and  the fact that  $\psi$  is regularly varying with index $\varrho/2\ge0$ by  Proposition 1.5.7 (i) in \cite{BinghamGoldieTeugels}, we have
\begin{equation}\label{psi/psi(log..) lower bound2}
\inf_{1-\eps\le\lambda\le 1+\eps}\frac{\psi(n)}{\psi(\lambda n)}
\lra (1-\eps)^{\varrho/2}.
\end{equation}
 To deal with the second factor on the right-hand side of \eqref{Z psi eq}, we choose  $n_6\ge n_5$ such that $\psi(n)\le (1+\eps)
  s_n$ for all $n\ge n_6$. Then, for $n\ge n_6$,
  \begin{equation}\label{|Zn|/psi(n) lower bound}
  \frac{ |Z_n|}{\psi(n)}\Ind{\{|Z_n|/\psi(n)\ge\eps\}}\ge \frac{ |Z_n|}{(1+\eps)s_n}\Ind{\{|Z_n|/s_n\ge\eps(1+\eps)\}}.
  \end{equation}
  Altogether, we get 
 \begin{align*}
    \liminf _{n \uparrow \infty} 2^{n} \mathbb{E}\left[{\Phi_0}\left(2^{-n}\left|Z_{n}\right|\right)\right] \geq \frac{ (1-\eps)^{\varrho/2}}{(1+\eps)\sqrt{2 \pi }} \int_{\{|z| \geq \varepsilon(1+\eps)\}}|z| e^{-z^{2} /2} \,d z.
\end{align*}
With \eqref{eqV}, we thus get  $\liminf_nV_n\ge\sqrt{2/\pi}$ by sending $\eps\da0$.

To get an upper bound, we recall  that $x\mapsto \Phi_0(1/x)$ is regularly varying (at infinity) with index $-1<0$. Theorem 1.8.2 in \cite{BinghamGoldieTeugels} hence yields a function $h$ that is regularly varying with index $-1$, strictly decreasing, and satisfies $\Phi_0(1/x)\le h(x)$ as well as $h(x)/\Phi_0(1/x)\to1$ as $x\ua\infty$. Thus, the function $\Psi(x):=h(1/x)$ is strictly increasing, regularly varying at zero with index  $1$, and satisfies $\Phi_0\le \Psi$ as well as $\Psi(x)/\Phi_0(x)\to1$ as $x\da0$.  
Let $\eps>0$ be given and choose $x_0>0$ such that $\Psi(x)\le \Phi_0(x)/(1-\eps)$ for all $x\le x_0$. 
%As a matter of fact, by replacing $g$ with $x\mapsto g(x\vee \wt x_0)$ for $\wt x_0:=-(1-q)^{-1}\log_2x_0$ and performing a similar operation on $\Psi$, we may assume without loss of generality that $\Psi(x)\le \Phi_0(x)/(1-\eps)$ for all $x<1$.
By \eqref{|Zn| eq}, there is $n_7\in\bN$ such that $2^{-n}|Z_n|\le x_0$ for all $n\ge n_7$. For such $n$, we get
\begin{align*}
2^{n} \mathbb{E}\left[\Phi_0\left(2^{-n}\left|Z_{n}\right|\right)\right]&\le 2^{n} \mathbb{E}\left[\Psi\left(2^{-n}\left|Z_{n}\right|\right)\right]\le 2^{n} \mathbb{E}\left[\Psi\left(2^{-n}(\eps+\left|Z_{n}\right|)\right)\right]\le \frac{2^{n}}{1-\eps} \mathbb{E}\left[\Phi_0\left(2^{-n}(\eps+\left|Z_{n}\right|)\right)\right]\\
&=\frac{s_n}{(1-\eps)\psi(n)}\bE\bigg[\frac{\eps+|Z_n|}{s_n}\cdot\frac{\psi(n)}{\psi(n-\log_2(\eps+|Z_n|))}\bigg].
\end{align*}
Clearly, the factor $s_n/\psi(n)$ converges to 1. Moreover, by \eqref{n-log2|Zn| lower bound} there is $n_8\ge n_5\vee n_7$ such that 
$$(1-\eps)n\le n-\log_2(\eps+|Z_n|)\le n-\log_2\eps\le (1+\eps)n\qquad\text{for all $n\ge n_8$.}
$$
As in \eqref{psi/psi(log..) lower bound} and \eqref{psi/psi(log..) lower bound2}, we thus get 
$$\frac{\psi(n)}{\psi(n-\log_2(\eps+|Z_n|))}\le \sup_{1-\eps\le\lambda\le 1+\eps}\frac{\psi(n)}{\psi(\lambda n)}\lra (1+\eps)^{\varrho/2}.
$$
We hence obtain the upper bound
\begin{equation}\label{Vn q=0 upper bound}
\limsup_{n\ua\infty}V_n\le \frac{ (1+\eps)^{\varrho/2}}{(1-\eps)\sqrt{2 \pi }} \int|z|e^{-z^{2} /2} \,d z,
\end{equation}
where the right-hand side reduces to $\sqrt{2/\pi}$ as $\eps\da0$.\end{proof}

The proof of  \Cref{Theorem on Phi-variation} {\rm(b)} for $0<t<1$ requires that we truncate the summation for $V_n$ in \eqref{eqV}
 at those indices $k$ for which $k2^{-n}\le t$. This is equivalent to restricting the expectation in \eqref{eqV}  to the set $\{2^{-n}R_n \le t\}$. The following proof adapts the arguments of the proof of  \Cref{Theorem on Phi-variation} {\rm(b)} for $t=1$ to this restricted case.

\begin{proof}[Proof of \Cref{Theorem on Phi-variation} {\rm(b)} for $0<t<1$] Let $ t \in(0, 1)$ be given. The $\Phi_0$-variation of $f$ over the interval $[0,t]$ is equal to $\lim_nV_{n,t}$, where
\begin{equation*}
\begin{split}
    V_{n,t}:= \sum_{k = 0}^{2^n-1}\Phi_0\Big(|f((k+1)2^{-n})-f(k2^{-n})|\Big)\mathbbm{1}_{[0,t]}(k2^{-n})= 2^n\bE\Big[\Phi_0\Big(2^{-n}|Z_n|\Big)\mathbbm{1}_{\{2^{-n}R_n \le t\}}\Big].
\end{split}
\end{equation*}
This can be proved similarly as in the derivation of \eqref{eqV}, see e.g. \cite{HanSchiedZhang1}. We fix $\delta\in(0,t\wedge(1-t))$ and $m \in \bN$ that $2^{-m} \le \delta$. Recall that $R_m := \sum_{k = 1}^{m}U_k2^{k-1}$ and let us denote $R_{m,n}:= R_n - R_{n-m}$ for short. Then
\begin{equation}\label{Rn inclusions eq}
\{2^{-n}R_{m,n} \le t -\delta\} \subseteq \{2^{-n}R_n \le t\} \subseteq \{2^{-n}R_{m,n} \le t\}.
\end{equation}
 First, we derive a lower bound for $V_{n,t}$. 
 To this end, we define for $0\le m<n$,
 $$Z_{m,n}:=\sum_{k=1}^{n-m}\beta_{n-k}Y_k=Z_n-\sum_{k=n-m+1}^n\beta_{n-k}Y_k.
 $$
 Then, since $s_n\to\infty$,
 $$\frac{|Z_n-Z_{m,n}|}{s_n}\le\frac1{s_n}\sum_{k=0}^{m-1}|\beta_{k}|=:a_n\lra0,
 $$
 and so the laws of $Z_{m,n}/s_n$ converge to $N(0,1)$ in the $L^1$-Wasserstein distance by Slutsky's theorem.
 By combining \eqref{Z psi eq}, \eqref{psi/psi(log..) lower bound}, \eqref{psi/psi(log..) lower bound2}, and \eqref{|Zn|/psi(n) lower bound}, there is  $n_9\ge n_6\vee m$ such that for all $n \ge n_9 $, with $c_\eps:=(1-\eps)^{1+\varrho/2}/(1+\eps)$,
\begin{equation*}
\begin{split}
    2^{n}\Phi_0(2^{-n}|Z_n|)\mathbbm{1}_{\{|Z_n|/\psi(n) \ge \eps,\,2^{-n}R_n \le t\}}&\ge \frac{c_\eps|Z_n|}{s_n}\mathbbm{1}_{\{|Z_n|/\psi(n) \ge \eps,\,2^{-n}R_n \le t\}}\\
  %&\ge \frac{|Z_{m,n}|}{(1+\eps)
  %s_n}\Ind{\{|Z_n|/\psi(n) \ge \eps,\,2^{-n}R_n \le t\}}-\frac{|Z_n-Z_{m,n}|}{(1+\eps)
 % s_n}\Ind{\{|Z_n|/\psi(n) \ge \eps,\,2^{-n}R_n \le t\}}
%  \\
  &\ge \frac{c_\eps|Z_{m,n}|}{
  s_n}\Ind{\{|Z_n|/\psi(n) \ge \eps,\,2^{-n}R_n \le t\}}-\frac{c_\eps|Z_n-Z_{m,n}|}{  s_n}\\ &\ge  \frac{c_\eps|Z_{m,n}|}{
  s_n}\Ind{\{|Z_n|/\psi(n) \ge \eps,\,2^{-n}R_n \le t\}}-\frac{c_\eps\sum_{k = 0}^{m-1}|\beta_k|}{
  s_n},
 \end{split}
\end{equation*}
where the last two inequalities hold due to the triangle inequality. Again,  the right-most term decays to zero as $n\ua\infty$. Hence, there is $n_{10} \ge n_9$ such that  for $n \ge  n_{10}$  and $|Z_{m,n}|/s_n \ge 2\eps(1+\eps)$, 
\begin{equation*}
    \frac{|Z_n|}{s_n} \ge \frac{|Z_{m,n}|-|Z_n-Z_{m,n}|}{s_n} \ge \frac{|Z_{m,n}|-\sum_{k = 0}^{m-1}|\beta_k|}{s_n} \ge \eps(1+\eps).
\end{equation*}
Hence, for those $n$ and $ m$, we have $\{|Z_{m,n}|/s_n \ge 2\eps(1+\eps)\}\subseteq\{|Z_n|/s_n \ge \eps(1+\eps)\}$ and in turn
\begin{equation*}
\begin{split}
   V_{n,t} +c_\eps a_n&\ge \bE\bigg[\frac{c_\eps|Z_{m,n}|}{
  s_n}\Ind{\{|Z_n|/\psi(n) \ge \eps,\,2^{-n}R_n \le t\}}\bigg]\ge \bE\bigg[\frac{c_\eps|Z_{m,n}|}{  s_n}\Ind{\{|Z_{m,n}|/s_n\ge 2\eps (1+\eps),\,2^{-n}R_{m,n} \le t -\delta\}}\bigg] \\
  &\ge \bE\bigg[\frac{c_\eps|Z_{m,n}|}{  s_n}\Ind{\{|Z_{m,n}|/s_n\ge2\eps (1+\eps)\}}\bigg]\cdot \bP\big[2^{-n}R_{n} \le t -\delta\big] ,
\end{split}
\end{equation*}
where the last step follows from \eqref{Rn inclusions eq} and the independence of $Z_{m,n}$ and $R_{m,n}$. Clearly, $\bP[2^{-n}R_{n} \le t -\delta] \rightarrow t-\delta$ and 
\begin{equation*}
    \bE\bigg[\frac{c_\eps|Z_{m,n}|}{
  s_n}\Ind{\{|Z_{m,n}|/s_n\ge2\eps (1+\eps)\}}\bigg] \longrightarrow \frac{c_\eps}{\sqrt{2\pi}}\int_{\{|z| \geq 2\varepsilon(1+\eps)\}}|z| e^{-z^{2} /2} \,d z,
\end{equation*}
since $s_{n-m}/s_n \rightarrow 1$ by \Cref{CLT cond prop}. Hence
\begin{align}\label{haha}
\liminf_{n\ua\infty}V_{n,t} \ge\frac{c_\eps(t-\delta)}{\sqrt{2\pi}}\int_{\{|z| \geq 2\varepsilon(1+\eps)\}}|z| e^{-z^{2} /2}\, d z.
\end{align}
To obtain a corresponding upper bound, we can argue exactly as in the derivation of \eqref{Vn q=0 upper bound} to get \begin{equation*}
\begin{split}
\limsup _{n \uparrow \infty}V_{n,t} \le  (1+\eps)^{\varrho/2}\limsup _{n \uparrow \infty}\bE\bigg[\frac{\eps+|Z_n|}{s_n}\mathbbm{1}_{\{2^{-n}R_{m,n} \le t\}}\bigg]=(1+\eps)^{\varrho/2}\limsup _{n \uparrow \infty}\bE\bigg[\frac{|Z_n|}{s_n}\mathbbm{1}_{\{2^{-n}R_{m,n} \le t\}}\bigg] .
\end{split}
\end{equation*}
Using again the independence of $Z_{m,n}$ and $R_{m,n}$, we find that
$$\bE[|Z_{n}|\mathbbm{1}_{\{2^{-n}R_{m,n} \le t\}}]\le \bE[|Z_{m,n}|]\cdot\bP[2^{-n}R_{m,n} \le t]+\sum_{k=0}^{m-1}|\beta_k|.
$$
By \eqref{Rn inclusions eq}, we have $\bP[2^{-n}R_{m,n} \le t]\le \bP[2^{-n}R_{n} \le t+\delta]\to t+\delta$. Recall that 
$\frac1{s_n}\sum_{k=0}^{m-1}|\beta_k|\to0$ as $n\ua\infty$.  Next, the function $\ell(x):=g(\log_2x)$ is slowly varying by \Cref{ell g remark}, and so 
\begin{equation}\label{g(n-m) eq}
\frac{g(n)}{g(n-m)}=\frac{\ell(2^n)}{\ell(2^{-m}2^n)}\lra1\qquad\text{as $n\ua\infty$.}
\end{equation}
Thus, our condition \eqref{Thm_eq_1} for $q=0$ implies that 
$$\bE\Big[\frac{|Z_{m,n}|}{s_n}\Big]=\frac{s_{n-m}}{s_n}\bE\Big[\frac{|Z_{m,n}|}{s_{n-m}}\Big]\lra\sqrt{\frac2\pi}\qquad\text{as $n\ua\infty$.}
$$
Altogether, we conclude that $\limsup_nV_{n,t}\le (1+\eps)^{\varrho/2}(t+\delta)\sqrt{2/\pi}$.  Combining this inequality with \eqref{haha} and sending $\eps$ and $\delta$ to zero, we conclude the proof of the linearity of the $\Phi_0$-variation.
\end{proof}

The proof of part (c) of  \Cref{Theorem on Phi-variation} will be prepared with the following two lemmas.

\begin{lemma}\label{Convergence_Lemma_2}
Under the conditions of Theorem \ref{Theorem on Phi-variation} (c),   there is a constant $K$ such that for all $\om\in\Om$ and all sufficiently large $n\in\bN$,
\[
2^{-qn}|Z_n(\om)|\leq K\sqrt{g(n)}\quad\text{and}\quad|Z_n(\om)|\leq Ks_n.
\]
\end{lemma}

\begin{proof} The function $\ell(x):=g(\log_2x)$ is slowly varying by \Cref{ell g remark}, and so $\ell$ and the sequence $\{\beta_m\}_{m \in \bN_0}$ satisfy the conditions of \Cref{conv thm}. 
Recall from \eqref{an2n ineq} in the proof of  that theorem that there exists $\delta_1>0$ such that for every $\delta\in(0,\delta_1)$ there is $n_1\in\bN$ such that  for all $m\ge n_1$,
 \begin{equation*}\label{an2n ineq2}
    2^{qm}\sqrt{(1-\delta)\ell(2^m) - (1 + \delta)\ell(2^{m-1})2^{-2q}} <|\beta_m|<  2^{qm}\sqrt{(1+\delta)\ell(2^m) - (1 - \delta)\ell(2^{m-1})2^{-2q}},
    \end{equation*}
and the arguments of the two square roots are positive. 
Now let $\delta\in(0,\delta_1)$ be given and take $n_2\ge n_1$ such that  $  1 - \delta < g(n)/g(n-1) < 1 + \delta$  for any $n \ge n_2$; this is possible by \eqref{g(n-m) eq}. 
Since $g$ is regularly varying, we have $2^{2qn}g(n)\to\infty$ as $n\ua\infty$ (see, e.g., Proposition B.1.9 in \cite{deHaanFerreira}). Therefore, we may apply the  Stolz--Ces\'{a}ro theorem in its general form so as to obtain 
\begin{align*}
    \limsup_{n \ua \infty} \dfrac{\sum_{m = 0}^n{|\beta_m|}}{2^{qn}\sqrt{g(n)}} &\le  \limsup_{n \ua \infty} \dfrac{{|\beta_n|}}{2^{qn}\sqrt{g(n)}-2^{q(n-1)}\sqrt{g(n-1)}}\\ 
    &\le \limsup_{n \ua \infty}\dfrac{\sqrt{(1+\delta)g(n) - (1 - \delta)g(n-1)2^{-2q}}}{\sqrt{g(n)}-\sqrt{2^{-2q}g(n-1)}}\\ 
      & \le \dfrac{\sqrt{(1+\delta)^22^{2q} - (1-\delta)}}{\sqrt{2^{2q}(1-\delta})-1}.
\end{align*}
Since $|Z_n|\le\sum_{m = 0}^{n-1}|\beta_m|$, the claim follows.
\end{proof}

\begin{lemma}\label{Phi conv lemma} Suppose that $\{z_n\}_{n \in \bN_0}$ is a sequence of nonnegative numbers converging to $z>0$. Then $2^n\Phi_q(2^{-n}s_nz_n)\to z^{1/(1-q)}$ as $n\ua\infty$.
\end{lemma}

\begin{proof}We may assume without loss of generality that $z_n>0$ for all $n$. Then we may write 
\begin{equation*}\label{Phisplit eq}
\begin{split}
2^n\Phi_q(2^{-n}s_nz_n)&=2^n\big(2^{-n}s_nz_n\big)^{1/(1-q)}g\Big(\frac{n-\log_2s_n-\log_2z_n}{1-q}\Big)^{-1/(2(1-q))}\\
&=\Big(\frac{s_n}{2^{qn}\sqrt{g(n)}}\Big)^{1/(1-q)}z_n^{1/(1-q)}\Bigg(\frac{g(n)}{g\big(\frac{n-\log_2s_n-\log_2z_n}{1-q}\big)}\Bigg)^{1/(2(1-q))}.
\end{split}
\end{equation*}
By \eqref{Thm_eq_1}, the first factor on the right converges to 1. The second factor converges   to $z^{1/(1-q)}$ by assumption.  To deal with the third factor, \eqref{Thm_eq_1} implies that there is $n_0\in\bN$ such that 
\begin{equation*}
qn-\log_2\sqrt{g(n)}-1\le \log_2s_n\le qn-\log_2\sqrt{g(n)}+1\qquad\text{for all $n\ge n_0$.}
\end{equation*}
Theorem 1.4.1 and Proposition 1.3.6 (i) in \cite{BinghamGoldieTeugels} give moreover that  there exists $\kappa^+,\kappa^-\in\bR$ such that $\kappa^-\log_2n\le \log_2\sqrt{g(n)}\le\kappa^+\log_2n$. Thus, if $n\ge n_0$ is sufficiently large such that $z/2\le z_n\le 2z$, then
\begin{equation}\label{zn eq}
n+\frac{-1-\log_2(2z)-\kappa^+\log_2n}{1-q}\le \frac{n-\log_2s_n-\log_2z_n}{1-q}\le n+\frac{1-\log_2(z/2)-\kappa^-\log_2n}{1-q}.
\end{equation}
For sufficiently large $n$, the center term can thus be expressed as $n\lambda_n$, where $\lambda_n\in[1/2,3/2]$ and $\lambda_n\to1$. The uniform convergence theorem for regularly varying functions (e.g.,~Theorem 1.5.2 in \cite{BinghamGoldieTeugels}) hence implies that 
$$\frac{g(n)}{g\big(\frac{n-\log_2s_n-\log_2z_n}{1-q}\big)}\longrightarrow 1
$$
as $n\ua\infty$. This concludes the proof.
\end{proof}

\begin{proof}[Proof of  \Cref{Theorem on Phi-variation} {\rm (c)}] We prove  part {\rm (c)} only for $t=1$. The extension to $0<t<1$ is almost verbatim identical to the one in part (b) and hence left to the reader.  We write 
$$\wh Z_n:=\frac1{s_n}\sum_{m=1}^{n}|\beta_{n-m}|Y_m.
$$
Then the law of $2^n\Phi_q(2^{-n}|Z_n|)$ is the same as that of $2^n\Phi_q(2^{-n}s_n|\wh Z_n|)$.  By \Cref{conv thm}, we have $|\wh Z_n|\to | Z|$ in $L^\infty$. Moreover, by either Corollary 6.6 or Remark 6.7 in \cite{DutkayJorgensen}, the law of $Z$ has no atoms.  In particular, we have $\bP[Z=0]=0$. \Cref{Phi conv lemma}  hence yields that  $\bP$-a.s. $2^n\Phi_q(2^{-n}s_n|\wh Z_n|)\to|Z|^{1/(1-q)}$.  Fatou's lemma thus gives immediately that 
$$\liminf_{n\ua\infty}V_n=\liminf_{n\ua\infty}2^n\bE[\Phi_q(2^{-n}|Z_n|)]=\liminf_{n\ua\infty}\bE[2^n\Phi_q(2^{-n}s_n|\wh Z_n|)]\ge \bE[|Z|^{1/(1-q)}]. 
$$

To get an upper bound, we argue as in the proof of \Cref{Theorem on Phi-variation} {\rm(b)} for $t=1$ that there exists an increasing function $\Psi$ that is regularly varying at zero with index $1/(1-q)$  and satisfies $\Phi_q(x)\le\Psi(x)$ for all sufficiently small $x$ as well as $\Psi(x)/\Phi_q(x)\to1$ as $x\da0$. In particular, we have $2^n\Psi(2^{-n}s_nz_n)\to z^{1/(1-q)}$ if $\{z_n\}_{n \in \bN_0}$ is a sequence of nonnegative numbers converging to $z>0$. 
Let $\eps>0$ be given. We choose $x_0>0$ such that $\Psi(x)\le\Phi_q(x)/(1-\eps)$ for all $x\le x_0$. Since the sequence $\{\wh Z_n\}$ is uniformly bounded, there is $n_0\in\bN$ such that $2^{-n}s_n(\eps+|\wh Z_n|)\le x_0$ $\bP$-a.s.~for all $n\ge n_0$. For such $n$, we hence have
$$V_n=2^n\bE[\Phi_q(2^{-n}s_n|\wh Z_n|)]\le 2^n\bE[\Psi(2^{-n}s_n|\wh Z_n|)]\le 2^n\bE[\Psi(2^{-n}s_n(\eps+|\wh Z_n|))]\le \frac{2^n}{1-\eps}\bE[\Phi_q(2^{-n}s_n(\eps+|\wh Z_n|))].
$$
Moreover, we see from \eqref{zn eq} that there exists $n_1\ge n_0$ such that the random variables $\Lambda_n$ defined through
$$\Lambda_n:=\frac{n-\log_2s_n-\log_2(\eps+|\wh Z_n|)}{n(1-q)}
$$
take values in $[1/2,3/2]$ for $n\ge n_1$ and satisfy $\Lambda_n\to1$ $\bP$-a.s. Hence, with $\varrho$ denoting the index of regular variation of $g$,
$$0\le \frac{g(n)}{g\big(\frac{n-\log_2s_n-\log_2(\eps+|\wh Z_n|)}{1-q}\big)}= \frac{g(n)}{g(n\Lambda_n)}\le\sup_{1/2\le\lambda\le 3/2}\frac{g(n)}{g(n\lambda)}\lra\Big(\frac32\Big)^\varrho
$$
according to the uniform convergence theorem for regularly varying functions (e.g.,~Theorem 1.5.2 in \cite{BinghamGoldieTeugels}). Hence,  \Cref{Phi conv lemma} and dominated convergence give that 
$$\limsup_{n\ua\infty}V_n\le \limsup_{n\ua\infty}\frac1{1-\eps}2^n\bE[\Phi_q(2^{-n}s_n(\eps+|\wh Z_n|))]= \frac1{1-\eps}\bE[(\eps+|Z|)^{1/(1-q)}].$$
Sending $\eps\da0$ gives the desired upper bound.\end{proof}

 \begin{proof}[Proof of \Cref{g limit cor}]  (a) Taking logarithms in \eqref{Thm_eq_1} and using the fact that $\frac1n\log_2g(n)\to0$, we get $\frac1n\log_2s_n\to q$ as $n\ua\infty$. Hence, the result follows from \Cref{Prop Khine}.
 
 (b) If $0<c<\infty$, then \eqref{Thm_eq_1} holds also if $g$ is replaced with $\wt g(x):=c$. Hence, the assertion follows immediately from \Cref{Theorem on Phi-variation}. If $c=0$, then for any $\eps>0$ there is $\delta>0$ such that 
 $$g\Big(\frac{-\log_2\eta}{1-q}\Big)<\eps\qquad\text{for all $\eta<\delta$.}
 $$
 Now it suffices to take $n_0\in\bN$ such that $|f((k+1)2^{-n})-f(k2^{-n})| < \delta$ for all $k$ and $n\ge n_0$ to obtain that 
 \begin{equation*}
    \sum_{k=0}^{\lfloor t2^n\rfloor}|f((k+1)2^{-n})-f(k2^{-n})|^{p} \le \eps^{p/2}\sum_{k=0}^{\lfloor t2^n\rfloor}\Phi_q\big(|f((k+1)2^{-n})-f(k2^{-n})|\big).
\end{equation*}
 Sending $\eps\da0$ gives the result. The case $c=\infty$ is obtained analogously.
  \end{proof}

%  \noindent{\bf Data availability statement:} No data were used in the preparation of this manuscript. The authors will be happy to share upon request the Mathematica code for the generation of \Cref{rho figure}.

\bibliographystyle{plain}
\bibliography{CTbook}

\begin{thebibliography}{10}

\bibitem{Allaartflexible}
Pieter~C. Allaart.
\newblock On a flexible class of continuous functions with uniform local
  structure.
\newblock {\em Journal of the Mathematical Society of Japan}, 61(1):237--262,
  2009.

\bibitem{AllaartKawamura}
Pieter~C. Allaart and Kiko Kawamura.
\newblock The {T}akagi function: a survey.
\newblock {\em Real Analysis Exchange}, 37(1):1--54, 2011.

\bibitem{BinghamGoldieTeugels}
N.~H. Bingham, C.~M. Goldie, and J.~L. Teugels.
\newblock {\em Regular variation}, volume~27 of {\em Encyclopedia of
  Mathematics and its Applications}.
\newblock Cambridge University Press, Cambridge, 1989.

\bibitem{ContPerkowski}
Rama Cont and Nicolas Perkowski.
\newblock Pathwise integration and change of variable formulas for continuous
  paths with arbitrary regularity.
\newblock {\em Trans. Amer. Math. Soc. Ser. B}, 6:161--186, 2019.

\bibitem{deHaanFerreira}
Laurens de~Haan and Ana Ferreira.
\newblock {\em Extreme value theory}.
\newblock Springer Series in Operations Research and Financial Engineering.
  Springer, New York, 2006.
\newblock An introduction.

\bibitem{DutkayJorgensen}
Dorin~Ervin Dutkay and Palle E.~T. Jorgensen.
\newblock Harmonic analysis and dynamics for affine iterated function systems.
\newblock {\em Houston J. Math.}, 33(3):877--905, 2007.

\bibitem{Faber07}
G.~{Faber}.
\newblock {Einfaches Beispiel einer stetigen nirgends differenzierbaren
  Funktion}.
\newblock {\em {Jahresber. Dtsch. Math.-Ver.}}, 16:538--540, 1907.

\bibitem{FoellmerIto}
Hans. F{\"o}llmer.
\newblock Calcul d'{I}t\^o sans probabilit\'es.
\newblock In {\em Seminar on {P}robability, {XV} ({U}niv. {S}trasbourg,
  {S}trasbourg, 1979/1980)}, volume 850 of {\em Lecture Notes in Math.}, pages
  143--150. Springer, Berlin, 1981.

\bibitem{FoellmerSchiedBuch}
Hans F\"{o}llmer and Alexander Schied.
\newblock {\em Stochastic finance. An introduction in discrete time}.
\newblock De Gruyter Graduate. De Gruyter, Berlin, fourth revised and extended
  edition, 2016.

\bibitem{Gantert}
Nina Gantert.
\newblock Self-similarity of {B}rownian motion and a large deviation principle
  for random fields on a binary tree.
\newblock {\em Probab. Theory Related Fields}, 98(1):7--20, 1994.

\bibitem{GatheralRosenbaum}
Jim Gatheral, Thibault Jaisson, and Mathieu Rosenbaum.
\newblock Volatility is rough.
\newblock {\em Quantitative Finance}, 18(6):933--949, 2018.

\bibitem{Gladyshev}
E.~G. Gladyshev.
\newblock A new limit theorem for stochastic processes with {G}aussian
  increments.
\newblock {\em Teor. Verojatnost. i Primenen}, 6:57--66, 1961.

\bibitem{HanSchiedZhang1}
Xiyue Han, Alexander Schied, and Zhenyuan Zhang.
\newblock A probabilistic approach to the {$\Phi$}-variation of classical
  fractal functions with critical roughness.
\newblock {\em Statist. Probab. Lett.}, 168:108920, 2021.

\bibitem{HataYamaguti}
Masayoshi Hata and Masaya Yamaguti.
\newblock The {T}akagi function and its generalization.
\newblock {\em Japan J. Appl. Math.}, 1(1):183--199, 1984.

\bibitem{KawadaKono}
Takayuki Kawada and Norio K\^{o}no.
\newblock On the variation of {G}aussian processes.
\newblock In {\em Proceedings of the {S}econd {J}apan-{USSR} {S}ymposium on
  {P}robability {T}heory ({K}yoto, 1972)}, pages 176--192. Lecture Notes in
  Math., Vol. 330, 1973.

\bibitem{KonoOscillation}
Norio K\^{o}no.
\newblock Oscillation of sample functions in stationary {G}aussian processes.
\newblock {\em Osaka Math. J.}, 6:1--12, 1969.

\bibitem{Kono}
Norio K{\^o}no.
\newblock On generalized {T}akagi functions.
\newblock {\em Acta Math. Hungar.}, 49(3-4):315--324, 1987.

\bibitem{MarcusRosenPhi}
Michael~B. Marcus and Jay Rosen.
\newblock {$\Phi$}-variation of the local times of symmetric {L}\'{e}vy
  processes and stationary {G}aussian processes.
\newblock In {\em Seminar on {S}tochastic {P}rocesses, 1992 ({S}eattle, {WA},
  1992)}, volume~33 of {\em Progr. Probab.}, pages 209--220. Birkh\"{a}user
  Boston, Boston, MA, 1993.

\bibitem{MarcusRosen}
Michael~B. Marcus and Jay Rosen.
\newblock {\em Markov processes, {G}aussian processes, and local times}, volume
  100 of {\em Cambridge Studies in Advanced Mathematics}.
\newblock Cambridge University Press, Cambridge, 2006.

\bibitem{MishuraSchied}
Yuliya Mishura and Alexander Schied.
\newblock Constructing functions with prescribed pathwise quadratic variation.
\newblock {\em J. Math. Anal. Appl.}, 442(1):117 -- 137, 2016.

\bibitem{MishuraSchied2}
Yuliya Mishura and Alexander Schied.
\newblock On (signed) {T}akagi--{L}andsberg functions: {$p$}th variation,
  maximum, and modulus of continuity.
\newblock {\em J. Math. Anal. Appl.}, 473(1):258--272, 2019.

\bibitem{Natanson}
I.~P. Natanson.
\newblock {\em Theory of functions of a real variable}.
\newblock Frederick Ungar Publishing Co., New York, 1955.
\newblock Translated by Leo F. Boron with the collaboration of Edwin Hewitt.

\bibitem{PeresSchlagSolomyak}
Yuval Peres, Wilhelm Schlag, and Boris Solomyak.
\newblock Sixty years of {B}ernoulli convolutions.
\newblock In {\em Fractal geometry and stochastics, {II}
  ({G}reifswald/{K}oserow, 1998)}, volume~46 of {\em Progr. Probab.}, pages
  39--65. Birkh\"auser, Basel, 2000.

\bibitem{SchiedTakagi}
Alexander Schied.
\newblock On a class of generalized {T}akagi functions with linear pathwise
  quadratic variation.
\newblock {\em J. Math. Anal. Appl.}, 433:974--990, 2016.

\bibitem{SchiedZZhang}
Alexander Schied and Zhenyuan Zhang.
\newblock On the {$p$}th variation of a class of fractal functions.
\newblock {\em Proc. Amer. Math. Soc.}, 148(12):5399--5412, 2020.

\bibitem{Takagi}
Teiji Takagi.
\newblock A simple example of the continuous function without derivative.
\newblock In {\em Proc. Phys. Math. Soc. Japan}, volume~1, pages 176--177,
  1903.

\bibitem{Taylor}
S.~J. Taylor.
\newblock Exact asymptotic estimates of {B}rownian path variation.
\newblock {\em Duke Math. J.}, 39:219--241, 1972.

\bibitem{ZhaoWoodroofeVolny}
Ou~Zhao, Michael Woodroofe, and Dalibor Voln\'{y}.
\newblock A central limit theorem for reversible processes with nonlinear
  growth of variance.
\newblock {\em J. Appl. Probab.}, 47(4):1195--1202, 2010.

\end{thebibliography}

\end{document}